\documentclass[12pt]{article}

\date{\today}

\PassOptionsToPackage{dvipsnames,svgnames,x11names}{xcolor}
\usepackage{xcolor}
\definecolor{labelkey}{rgb}{0,0.08,0.45}
\definecolor{refkey}{rgb}{0,0.6,0.0}
\definecolor{Brown}{rgb}{0.45,0.0,0.05}
\definecolor{lime}{rgb}{0.00,0.8,0.0}
\definecolor{lblue}{rgb}{0.5,0.5,0.99}
\definecolor{OliveGreen}{rgb}{0,0.6,0}
\definecolor{tyrianpurple}{rgb}{0.4, 0.01, 0.24}
\definecolor{myseagreen}{HTML}{3FBC9D}
\definecolor{myblue}{rgb}{0.9,0.9,0.98}
\colorlet{hlcyan}{cyan!30}

\usepackage{mathpazo} % With old-style figures and real smallcaps
\usepackage[T1]{fontenc}
\usepackage{mathtools}
\usepackage{amssymb}
\usepackage{stmaryrd}
\usepackage{empheq}

\usepackage{hyperref}
\hypersetup{
  colorlinks=true,
  linkcolor=refkey,
  urlcolor=lblue,
  citecolor=magenta
}

\usepackage{amsthm}
\makeatletter
\def\th@plain{%
  \thm@notefont{}%
  \itshape % body font
}
\def\th@definition{%
  \thm@notefont{}%
  \normalfont % body font
}

\usepackage[capitalize,nameinlink]{cleveref}
\makeatother
\newtheorem{theorem}{Theorem}[section]
\newtheorem{lemma}[theorem]{Lemma}
\newtheorem{corollary}[theorem]{Corollary}
\newtheorem{proposition}[theorem]{Proposition}
\newtheorem{definition}[theorem]{Definition}

\newtheorem{fact}[theorem]{Fact}
\newtheorem{remark}[theorem]{Remark}
\crefname{theorem}{Theorem}{Theorems}
\Crefname{theorem}{Theorem}{Theorems}
\crefname{fact}{Fact}{facts}
\Crefname{fact}{Fact}{facts}
\crefname{equation}{}{equations}
\crefname{chapter}{Appendix}{chapters}
\crefname{item}{}{items}
\crefname{enumi}{}{}

\usepackage{nicematrix}
\usepackage{booktabs}
\usepackage{siunitx}
\usepackage{tikz}
\usepackage[color=myseagreen]{todonotes}
\usepackage{soul}
\usepackage{nameref}
\usepackage{graphicx}
\usepackage{float}
\usepackage[shortlabels,inline]{enumitem}
\setlist[enumerate]{nosep}
\usepackage{relsize}

\usepackage[margin=1in,footskip=0.25in]{geometry}
\makeatletter

\let\orig@label\label

\renewcommand{\label}[1]{%
  \begingroup
  \def\@currentlabelname{}%
  \ifx\current@theorem\relax\else
    \def\@currentlabelname{\current@theorem}%
  \fi
  \ifx\cref@currentlabel\undefined\else
    \let\@currentlabelname\cref@currentlabel
  \fi
  \orig@label{#1}%
  \endgroup
}

\AtBeginEnvironment{theorem}{\def\current@theorem{theorem}}
\AtBeginEnvironment{proposition}{\def\current@theorem{proposition}}
\AtBeginEnvironment{fact}{\def\current@theorem{fact}}
\makeatother

\newcommand{\seppthree}{\setlength{\itemsep}{-3pt}}

\newcommand{\lev}[1]{\ensuremath{\mathrm{lev}_{\leq #1}\:}}

\DeclareMathOperator*{\argmin}{argmin}

\newcommand{\Id}{\ensuremath{\operatorname{Id}}}

\newcommand{\ds}{\displaystyle}

\providecommand{\norm}[1]{\lVert#1\rVert}

\providecommand{\innp}[1]{\langle#1\rangle}

\allowdisplaybreaks

\author{
  Heinz H.\ Bauschke\thanks{
    Mathematics, University of British Columbia,
    Kelowna, B.C.\ V1V~1V7, Canada. E-mail: \texttt{heinz.bauschke@ubc.ca}.}
  ~~~and~~~
  Tran Thanh Tung\thanks{
    Mathematics, University of British Columbia,
    Kelowna, B.C.\ V1V~1V7, Canada. E-mail: \texttt{tung.tran@ubc.ca}.}
}

\title{\textsf{On the boundedness of the sequence generated by minibatch stochastic gradient descent}}

\begin{document}

\maketitle

\begin{abstract}
Stochastic Gradient Descent (SGD) with Polyak's stepsize has recently gained renewed attention in stochastic optimization. Recently, Orvieto, Lacoste-Julien, and Loizou introduced a decreasing variant of Polyak's stepsize, where convergence relies on a boundedness assumption of the iterates. They established that this assumption holds under strong convexity. In this paper, we extend their result by proving that boundedness also holds for a broader class of objective functions, including coercive functions. We also present a case in which boundedness may or may not hold.
\end{abstract}



\section{Introduction}
Let $X$ be a Euclidean space with inner product $\innp{x,y}$ and induced norm $\norm{x}:=\sqrt{\innp{x,x}}$. Our goal is to solve the optimization problem
\begin{equation*}
	\min_{x\in X} f(x):=\frac{1}{N}\sum_{i=1}^{N} f_i(x),
\end{equation*}
where each $f_i$ is convex, $L_i-$smooth, and bounded below. We denote $\ds L_{\max}:=\max_{i=1,\dots,N}L_{i}$. We use the notation $\ds\mu:=\inf f(X)$ for the infimum of any function $f$ throughout this paper. 

A popular algorithm for this setting, especially when $N$ is large, is the \emph{Stochastic Gradient Descent \cref{SGD}} algorithm
\begin{equation}
	x_{k+1}:=x_k-\gamma_k\nabla f_{B_k}(x_k),\label{SGD}\tag{SGD}
\end{equation}
where $\gamma_k>0$ is the stepsize, $B_k\subseteq\left\{1,\dots,N\right\}$ is a random batch with fixed size $b\geq1$, sampled independently and uniformly at each iteration, and $f_{B_k}:=\frac{1}{b}\sum_{i\in B_k} f_i$.

In \cite{Loizou2021}, Loizou, Vaswani, Hadj Laradji, and Lacoste-Julien proposed the following \textit{Stochastic Polyak's stepsize} \cref{SPS} for the \cref{SGD} when $b=1$:
\begin{equation}
	\gamma_k:=\lambda\min\left\{\frac{f_{i_k}(x_k)-\mu_{i_k}}{\norm{\nabla f_{i_k}(x_k)}^2},\frac{\gamma_{-1}}{\lambda}\right\},\label{SPS}\tag{SPS}
\end{equation}
where $\mu_i:=\inf f_i(X)$, and where $\gamma_{-1},\lambda>0$ are user-defined constants. However, \cref{SPS} requires the exact knowledge of $\mu_{i_k}$. The authors proved the following result:

\begin{fact}[{\cite[Theorem 3.4]{Loizou2021}}]\label{fact1.1}
	Let $f_i:X\to\mathbb{R}$ be convex, $L_i-$smooth, and bounded below for all $i=1,\dots,N$. Suppose that $f:=\frac{1}{N}\sum_{i=1}^Nf_i$ has a minimizer $x^*\in X$. Let $x_0\in X$, $\gamma_{-1}>0$ and $\lambda=1$. Then, the sequence $\left(x_k\right)_{k\in\mathbb{N}}$ generated by \cref{SGD} with \cref{SPS} satisfies
	\begin{equation*}
		\mathbb{E}\left[f(\bar{x}_k)-\mu\right]\leq\frac{\norm{x_0-x^*}^2}{\alpha k}+\frac{2\sigma^2\gamma_{-1}}{\alpha},
	\end{equation*}
	where $\alpha:=\min\big\{\frac{1}{2L_{\max}},\gamma_{-1}\big\}$, $\sigma^2:=\mu-\mathbb{E}\left[\mu_i\right]$, and $\bar{x}_k:=\frac{1}{k}\sum\limits_{i=0}^{k-1}x_i$.
\end{fact}

\begin{remark}
	\cref{fact1.1} implies that when $k\to\infty$, we get 
	\begin{equation*}
		\lim\limits_{k\to\infty}\mathbb{E}\left[f(\bar{x}_k)\right]\leq\mu+\frac{2\sigma^2\gamma_{-1}}{\alpha},
	\end{equation*}
	which the authors call ``convergence to a neighborhood of the solution''.
\end{remark}

In \cite{Orvieto2022}, Orvieto, Lacoste-Julien, and Loizou introduced the \textit{Decreasing Stochastic Polyak's Stepsize} \cref{DecSPS}
\begin{equation}
	\gamma_k:=\lambda_k\min\left\{\frac{f_{B_k}(x_k)-\underline{\mu_{B_k}}}{\norm{\nabla f_{B_k}(x_k)}^2},\frac{\gamma_{k-1}}{\lambda_{k-1}}\right\},\label{DecSPS}\tag{DecSPS}
\end{equation}
where $\underline{\mu_{B}}$ is a lower bound of $f_{B}$,  $\left(\lambda_k\right)_{k\in\mathbb{N}}$ is a decreasing sequence in $\mathbb{R}_{++}$, $\lambda_{-1}:=\lambda_0$ and $\gamma_{-1}>0$. Observe that for this stepsize, one does not need precise information about the optimal value of the batch objective function $f_{B_k}$ at each step but only a lower bound of $f_{B_k}$.

\begin{remark}[{\cite[Lemma 1]{Orvieto2022}}]
	From the definition of \cref{DecSPS} and the fact that $\left(\lambda_k\right)_{k\in\mathbb{N}}$ is a decreasing sequence, we have
	\begin{equation*}
			\left(\forall k\in\mathbb{N}\right)\quad \gamma_{k}\leq\lambda_k\frac{\gamma_{k-1}}{\lambda_{k-1}}\leq\gamma_{k-1},
		\end{equation*}
	i.e., $\left(\gamma_{k}\right)_{k\in\mathbb{N}}$ is indeed a decreasing stepsize.
\end{remark}

The main result for \cref{SGD} with \cref{DecSPS} is as follows:

\begin{fact}[{\cite[Theorem 3]{Orvieto2022}}]\label{fact1.4}
	Let $f_i:X\to\mathbb{R}$ be convex, $L_i-$smooth, and bounded below for all $i=1,\dots,N$. Suppose that $f:=\frac{1}{N}\sum_{i=1}^Nf_i$ has a minimizer $x^*\in X$. Let $\left(\lambda_k\right)_{k\in\mathbb{N}}$ be a decreasing sequence such that $\lambda_k\leq1$ for all $k\in\mathbb{N}$ and let $\lambda_{-1}:=\lambda_{0}$, $\gamma_{-1}>0$, $x_0\in X$. Let $\left(x_k\right)_{k\in\mathbb{N}}$ be a sequence generated by \cref{SGD} with \cref{DecSPS}. Assume that $M:=\sup_{k\in\mathbb{N}} \norm{x_k-x^*}^2<\infty$. Then we have for all $k\in\mathbb{N}$,
	\begin{equation*}
		\mathbb{E}\left[f(\bar{x}_k)\right]-\mu\leq\frac{M}{\alpha}\frac{1}{\lambda_{k-1}k}+\sigma_b^2\frac{\sum_{i=0}^{k-1}\lambda_{i}}{k},
	\end{equation*}
	where $\alpha:=\min\big\{\frac{1}{2L_{\max}},\frac{\gamma_{-1}}{\lambda_0}\big\}$ and $\bar{x}_k:=\frac{1}{k}\sum_{i=0}^{k-1}x_i$.
\end{fact}

\begin{fact}[Ces\`aro means]\label{fact1.5}
	If $\left(\lambda_k\right)_{k\in\mathbb{N}}$ lies in $\mathbb{R}$ and $\lambda_{k}\to\lambda$, then $\frac{1}{k}\sum_{i=0}^{k-1}\lambda_i\to\lambda$.
\end{fact}

\begin{remark}
	In light of \cref{fact1.4} and \cref{fact1.5}, if $k\lambda_{k-1}\to\infty$, and $\lambda_{k}\to0$ ($\Rightarrow\frac{1}{k}\sum_{i=0}^{k-1}\lambda_i\to0$), then $\lim_{k\to\infty}\mathbb{E}\left[f(\bar{x}_k)\right]=\mu$. A standard choice for $\left(\lambda_{k}\right)_{k\in\mathbb{N}}$ that satisfies these assumptions is $\lambda_k:=\frac{1}{(k+1)^\theta}$ with $\theta\in(0,1)$, or $\lambda_k:=\frac{\ln(k+2)}{(k+2)^\theta}$ with $\theta\in(0,1]$.
\end{remark}

\cref{fact1.4} relies on a crucial assumption that the sequence $(x_k)_{k\in\mathbb{N}}$ is bounded, which the author only verified when $\lambda_{k}:=\frac{1}{\sqrt{k+1}}$ and under the rather restrictive assumption that each $f_i$ is strongly convex (see \cite[Proposition 1]{Orvieto2022}). In fact, many recent variants of \cref{SGD} with \cref{DecSPS} still use this assumption (see \cite[Theorem 5]{Islamov2024}, \cite[Corollary 5]{Jiang2023}, \cite[Theorem 3.6]{Oikonomou2024}).

In this paper, we provide the following new results on \cref{fact1.4}:
\begin{itemize}
    \item We present a scheme in which the sequence generated by \cref{SGD} with \cref{DecSPS} can be unbounded (see \cref{theorem3.4}).
    \item We then relax the condition of \cite[Proposition 1]{Orvieto2022} from strong convexity to a broader class of functions whose lower level sets, with the minimizers removed, are bounded (e.g., coercive functions), and extend the result to both constant and diminishing step sizes (see \cref{theorem3.8}).
    \item Finally, we discuss cases not covered by the above results.
\end{itemize}

The rest of the paper is organized as follows: \cref{s:prelim} introduces preliminary results. \cref{s:counterex} presents an example in which the sequence generated by \cref{SGD} with \cref{DecSPS} blows up to infinity. \cref{s:main} establishes our main boundedness result. \cref{s:discussion} discusses scenarios not covered by our assumptions.

The notation we employ is standard and largely follows, e.g., \cite{Bauschke2017} and \cite{Garrigos2024}.

\section{Preliminary results}
\label{s:prelim}
In this section, we introduce definitions, facts, and lemmas from convex analysis that will be useful in the subsequent sections.

\begin{definition}[{\cite[Definition 5.1]{Beck2017}}]
	Let $L\geq0$.  A function $f: X\to\mathbb{R}$ is said to be {\rm$L$-smooth} if it is differentiable and satisfies
	\begin{equation*}
		\left(\forall x,y\in X\right)\quad\norm{\nabla f(x)-\nabla f(y)}\leq L\norm{x-y}.
	\end{equation*}
\end{definition}

\begin{remark}\label{remark2.2}
	In the finite sum setting, if $f:=\frac{1}{N}\sum^{N}_{i=1}f_i$ and each $f_i$ is $L_i$-smooth, then every batch function $f_B$ is $L_{\max}$-smooth.
\end{remark}

\begin{definition}[{\cite[Equation (2.5)]{Beck2017}}]
	A function $f:X\to\mathbb{R}$ is said to be {\rm convex} if
	\begin{equation*}
		\left(\forall x,y\in X\right)\left(\forall \lambda\in(0,1)\right)\quad f\big((1-\lambda)x+\lambda y\big)\leq\left(1-\lambda\right)f(x)+\lambda f(y).
	\end{equation*}
\end{definition}

\begin{remark}\label{remark2.4}
	In the finite sum setting, if $f:=\frac{1}{N}\sum^{N}_{i=1}f_i$ and each $f_i$ is convex, then every batch function $f_B$ is also convex.
\end{remark}

\begin{fact}[{\cite[Theorem 5.8]{Beck2017}}]\label{fact2.5}
	Let $L>0$ and let $f: X\to\mathbb{R}$ be convex and differentiable. Then $f$ is $L$-smooth if and only if
	\begin{equation*}
		\left(\forall x,y\in X\right)\quad f(y)\leq f(x)+\innp{\nabla f(x),y-x}+\frac{L}{2}\norm{y-x}^2.
	\end{equation*}
\end{fact}

\begin{lemma}[{\cite[Lemma 2.28]{Garrigos2024}}]\label{lemma2.6}
	Let $L>0$ and let $f: X\to\mathbb{R}$ be convex, $L$-smooth and bounded below. Then
	\begin{equation*}
		\left(\forall x\in X\right)\quad f(x)-\mu\geq\frac{1}{2L}\norm{\nabla f(x)}^2.
	\end{equation*}
	Note that $f$ does not need to have a minimizer.
\end{lemma}

% \begin{proof}
% 	Let $x\in X$ be arbitrary. Fact \ref{fact2.5} yields
% 	\begin{align*}
% 		\left(\forall y\in X\right)\quad \mu&\leq f(x)+\innerproduct{\nabla f(x)}{y-x}+\frac{L}{2}\norm{y-x}^2\\
% 		&=f(x)+\norm{\sqrt{\frac{L}{2}}(y-x)+\sqrt{\frac{1}{2L}}\nabla f(x)}^2-\frac{1}{2L}\norm{\nabla f(x)}^2.
% 	\end{align*}
% 	Letting $y = x - \frac{1}{L}\nabla f(x)$ concludes the proof.
% \end{proof}

\begin{definition}[{\cite[Definition 11.11]{Bauschke2017}}]
	Let $f:X\to\mathbb{R}$. Then $f$ is {\rm coercive} if
	\begin{equation*}
		\lim_{\norm{x}\to\infty}f(x)=\infty.
	\end{equation*}
\end{definition}

\begin{definition}[{\cite[Definition 1.4]{Bauschke2017}}]
	Let $f:X\to\mathbb{R}$. The {\rm lower level set} of $f$ at height $\xi\in\mathbb{R}$ is
	\begin{equation*}
		\lev{\xi}f:=\left\{x\in X\mid f(x)\leq\xi\right\}.
	\end{equation*}
\end{definition}

\begin{fact}[{\cite[Proposition 11.12]{Bauschke2017}}]\label{fact2.9}
	Let $f:X\to\mathbb{R}$. Then $f$ is coercive if and only if its lower level sets $\left(\lev{\xi}f\right)_{\xi\in\mathbb{R}}$ are bounded.
\end{fact}

\begin{fact}[{\cite[Proposition 11.13]{Bauschke2017}}]\label{fact2.10}
	Let $f:X\to\mathbb{R}$ be convex. Then $f$ is coercive if and only if there exists $\xi\in\mathbb{R}$ such that $\lev{\xi}f$ is nonempty and bounded.
\end{fact}

\begin{fact}[{\cite[Proposition 1.2.3]{Bertsekas2016}}]\label{fact2.11}
	Assume that $f$ is $L$-smooth and bounded below. Suppose also that $\gamma_{k}\to0$, $\sum_{k\in\mathbb{N}}\gamma_{k}=\infty$, and $\gamma_k>0$ for all $k\in\mathbb{N}$. Let $\left(x_k\right)_{k\in\mathbb{N}}$ be a sequence generated by a gradient method $x_{k+1}=x_k-\gamma_{k}\nabla f(x_k)$. Then, $\left(f(x_k)\right)_{k\in\mathbb{N}}$ converges to a finite value and $\nabla f(x_k)\to0$. Furthermore, every limit point of $\left(x_k\right)_{k\in\mathbb{N}}$ is a stationary point of $f$.
\end{fact}

\section{Main results}
Recall the \cref{SGD} algorithm with fixed batch size $b\geq1$:
\begin{equation*}
	x_{k+1}:=x_k-\gamma_k\nabla f_{B_k}(x_k),
\end{equation*}
where $\gamma_k>0$. We consider 3 cases:

\begin{enumerate}[label=(C\arabic*)]
	\item $\left(\exists B\subseteq\left\{1,\dots,N\right\},|B|=b\right)\quad\argmin f_B=\varnothing$.\label{C1}
	\item $\left(\forall B\subseteq\left\{1,\dots,N\right\},|B|=b\right)\quad\argmin f_B\neq\varnothing$ and $\lev{\xi}f_B\setminus\argmin f_B$ is bounded for all $\xi\in\mathbb{R}$.\label{C2}
	\item $\left(\forall B\subseteq\left\{1,\dots,N\right\},|B|=b\right)\quad\argmin f_B\neq\varnothing$, but there exists a batch $\hat{B}$ and a constant $\hat{\xi}\in\mathbb{R}$ such that $\lev{\hat{\xi}}f_{\hat{B}}\setminus\argmin f_{\hat{B}}$ is unbounded.\label{C3}
\end{enumerate}

\begin{remark}
	A sufficient condition for \ref{C2} is that the number of coercive functions satisfies
	\begin{equation*}
		\#\left\{i\in\left\{1,\dots,N\right\}\mid f_i\text{ is coercive}\right\}\geq N-b+1;
	\end{equation*}
	indeed, in this case, every batch function is the average of at least one coercive function and several bounded below functions, and is therefore coercive. The conclusion then follows from \cref{remark2.4}, \cref{fact2.9} and \cref{fact2.10}.
\end{remark}

\begin{remark}
	When $X=\mathbb{R}$, \ref{C3} will never happen: indeed, if $\argmin f_{\hat{B}}$ were bounded, \cref{fact2.9} and \cref{fact2.10} would lead to a contradiction. WLOG, assume $\argmin f_{\hat{B}}=(-\infty,a]$. Since $\lev{\hat{\xi}}f_{\hat{B}}\setminus\argmin f_{\hat{B}}$ is unbounded and $\argmin f_{\hat{B}}\subseteq\lev{\hat{\xi}}f_{\hat{B}}$, it follows that $\lev{\hat{\xi}}f_{\hat{B}}=\mathbb{R}$; that is, $f_{\hat{B}}$ is bounded above. This, combined with \cref{remark2.4} and \cite[Corollary 8.6.2]{Rockafellar1997}, implies that $f_{\hat{B}}$ is a constant function. However, this would mean $\lev{\hat{\xi}}f_{\hat{B}}\setminus\argmin f_{\hat{B}}=\varnothing$, which is a contradiction.
\end{remark}

\begin{fact}[{\cite[Lemma 1]{Orvieto2022}}]\label{fact3.3}
	Let $f_i:X\to\mathbb{R}$ be convex, $L_i-$smooth, and bounded below for all $i=1,\dots,N$. Let $\left(\lambda_k\right)_{k\in\mathbb{N}}$ be a decreasing sequence in $\mathbb{R}_{++}$. Let $\lambda_{-1}:=\lambda_0$, $\gamma_{-1}>0$ and let $\left(\gamma_k\right)_{k\in\mathbb{N}}$ be the stepsize sequence generated by \cref{SGD} with \cref{DecSPS}. Then,
	\begin{equation*}
		\left(\forall k\in\mathbb{N}\right)\quad\min\left\{\frac{1}{2L_{\max}},\frac{\gamma_{-1}}{\lambda_0}\right\}\lambda_k\leq\gamma_k\leq \frac{\gamma_{-1}}{\lambda_0}\lambda_k.
	\end{equation*}
\end{fact}

\subsection{Case \ref{C1}}
\label{s:counterex}
\begin{theorem}\label{theorem3.4}
	Let $f_i:X\to\mathbb{R}$ be convex, $L_i-$smooth, and bounded below for all $i=1,\dots,N$. Suppose that \ref{C1} is satisfied. Let $\left(\lambda_k\right)_{k\in\mathbb{N}}$ be a decreasing sequence such that $\lambda_{k}\to0$, $\sum_{k\in\mathbb{N}}\lambda_{k}=\infty$ and $0<\lambda_0\leq1$. Let $\lambda_{-1}:=\lambda_{0}$, $\gamma_{-1}>0$, $x_0\in X$. Then, there exists a sequence of batches $\left(B_k\right)_{k\in\mathbb{N}}$ such that the corresponding sequence $\left(x_k\right)_{k\in\mathbb{N}}$ generated by \cref{SGD} with \cref{DecSPS} satisfies $\lim_{k\to\infty} \norm{x_k}=\infty$.
\end{theorem}

\begin{proof}
	Let $\hat{B}$ be the batch where $\argmin f_{\hat{B}}(X)=\varnothing$. Consider the sequence $\left(y_k\right)_{k\in\mathbb{N}}$ generated by the gradient method applied to the problem $\min_{x\in X}f_{\hat{B}}(x)$, using the stepsize defined in \cref{DecSPS} but restricted only to the batch $\hat{B}$, that is,
	\begin{equation*}
		\gamma_{k,\hat{B}}:=\lambda_k\min\left\{\frac{f_{\hat{B}}(x_k)-\underline{\mu_{\hat{B}}}}{\norm{\nabla f_{\hat{B}}(x_k)}^2},\frac{\gamma_{k-1,\hat{B}}}{\lambda_{k-1}}\right\}.
	\end{equation*}
	Then, $(\gamma_{k,\hat{B}})_{k\in\mathbb{N}}$ is a deterministic sequence and it is bounded: indeed, by \cref{fact3.3},
	\begin{equation*}
		\left(\forall k\in\mathbb{N}\right)\quad\min\left\{\frac{1}{2L_{\max}},\frac{\gamma_{-1,\hat{B}}}{\lambda_0}\right\}\lambda_k\leq\gamma_{k,\hat{B}}\leq \frac{\gamma_{-1,\hat{B}}}{\lambda_0}\lambda_k.
	\end{equation*}
	Because $\lambda_{k}\to0$ and $\sum_{k\in\mathbb{N}}\lambda_{k}=\infty$, we have $\gamma_{k,\hat{B}}\to0$ and $\sum_{k\in\mathbb{N}}\gamma_{k,\hat{B}}=\infty$. By \cref{remark2.2} and \cref{fact2.11}, every limit point of $\left(y_k\right)_{k\in\mathbb{N}}$ is a stationary point of $f_{\hat{B}}$. However, since $f_{\hat{B}}$ is convex (see \cref{remark2.4}) and has no minimizer, it cannot have a stationary point. Thus, $\norm{y_k}\to\infty$ as $k\to\infty$. Moreover, if we were to sample the same batch $\hat{B}$ at every step, then $\left(x_k\right)_{k\in\mathbb{N}}$ would coincide with $\left(y_k\right)_{k\in\mathbb{N}}$, which implies $\lim_{k\to\infty} \norm{x_k}=\infty$.
\end{proof}

\subsection{Case \ref{C2}}
\label{s:main}
Assume that there exists $m\in\mathbb{R}_+$ such that for every batch $B\subseteq\left\{1,\dots,N\right\}$ with $|B|=b$, there exists a lower bound $\underline{\mu_{B}}$ of $f_B$ such that the stepsize sequence $\left(\gamma_{k}\right)_{k\in\mathbb{N}}$ of \cref{SGD} satisfies
\begin{equation}
	\left(\forall k\in\mathbb{N}\right)\quad\gamma_k\norm{\nabla f_{B_k}(x_k)}^2\leq m \big(f_{B_k}(x_k)-\underline{\mu_{B_k}}\big).\label{3.1}
\end{equation}

\begin{remark}\label{remark3.5}
	If $f_i:X\to\mathbb{R}$ is convex, $L_i-$smooth, and bounded below for all $i=1,\dots,N$, then \cref{remark2.2} , \cref{remark2.4}, and \cref{lemma2.6} imply that any sequence $\left(\gamma_{k}\right)_{k\in\mathbb{N}}$ that is bounded above by some $\gamma\in\mathbb{R}_+$ satisfies \cref{3.1} with $m=2L_{\max}\gamma$.
\end{remark}

\begin{remark}\label{remark3.6}
	From the definition of \cref{DecSPS}, we have
	\begin{equation*}
		\gamma_k\overset{\text{def}}{=}\lambda_k\min\left\{\frac{f_{B_k}(x_k)-\underline{\mu_{B_k}}}{\norm{\nabla f_{B_k}(x_k)}^2},\frac{\gamma_{k-1}}{\lambda_{k-1}}\right\}\leq\lambda_{k}\frac{f_{B_k}(x_k)-\underline{\mu_{B_k}}}{\norm{\nabla f_{B_k}(x_k)}^2},
	\end{equation*}
	and so,
	\begin{equation*}
		\gamma_{k}\norm{\nabla f_{B_k}(x_k)}^2\leq\lambda_{k}\big(f_{B_k}(x_k)-\underline{\mu_{B_k}}\big).
	\end{equation*}
	Since $\left(\lambda_{k}\right)_{k\in\mathbb{N}}$ is a decreasing sequence, \cref{3.1} is satisfied with $m=\lambda_0$.
\end{remark}

It is discussed in \cite[Remark 1]{Orvieto2022} that for \cref{DecSPS}, if at the $k$th step we get $\nabla f_{B_k}(x_k)=0$, we simply resample a different batch, as the update direction in \cref{SGD} would otherwise be zero, making the choice of stepsize irrelevant. Moreover, the authors stated that this event never occurred in their experiments. Therefore, we will assume that this situation never happens; that is, $x_k\notin\argmin f_{B_k}$ for all $k\in\mathbb{N}$.

\begin{proposition}\label{proposition3.7}
	Let $f_i:X\to\mathbb{R}$ be convex, differentiable, and bounded below for all $i=~1,\dots,N$. Let $x_0\in X$ and let $\left(x_k\right)_{k\in\mathbb{N}}$ be a sequence generated by \cref{SGD} with stepsize $\left(\gamma_{k}\right)_{k\in\mathbb{N}}$ satisfying \cref{3.1}. Then,
	\begin{equation*}
		\left(\forall k\in\mathbb{N}\right)\quad\norm{x_{k+1}}^2\leq\norm{x_k}^2-\left(2-m\right)\gamma_k\big(f_{B_k}(x_k)-\underline{\mu_{B_k}}\big)+2\gamma_{k}\big(f_{B_k}(0)-\underline{\mu_{B_k}}\big).
	\end{equation*}
\end{proposition}

\begin{proof}\belowdisplayskip=-13pt
	From \cref{SGD}, \cref{3.1} and convexity of the batch function $f_{B_k}$ (see \cref{remark2.4}), we obtain
        \begin{align}
		\norm{x_{k+1}}^2\overset{\cref{SGD}}{=}\, \ \ &\norm{x_k-\gamma_{k}\nabla f_{B_k}(x_k)}^2\notag\\
		=\ \ \ \ \ &\norm{x_k}^2-2\gamma_{k}\innp{\nabla f_{B_k}(x_k),x_k}+\gamma_{k}^2\norm{\nabla f_{B_k}(x_k)}^2\notag\\
		\overset{\cref{3.1}}{\leq}\ \ \ \ \ &\norm{x_k}^2-2\gamma_{k}\innp{\nabla f_{B_k}(x_k),x_k}+m\gamma_{k}\big(f_{B_k}(x_k)-\underline{\mu_{B_k}}\big)\notag\\
		\overset{\text{convexity}}{\leq}&\norm{x_k}^2-2\gamma_{k}\big(f_{B_k}(x_k)-f_{B_k}(0)\big)+m\gamma_{k}\big(f_{B_k}(x_k)-\underline{\mu_{B_k}}\big)\notag\\
		=\ \ \ \ \ &\norm{x_k}^2-2\gamma_{k}\big(f_{B_k}(x_k)-\underline{\mu_{B_k}}+\underline{\mu_{B_k}}-f_{B_k}(0)\big)\notag\\
		&+m\gamma_{k}\big(f_{B_k}(x_k)-\underline{\mu_{B_k}}\big)\notag\\
		=\ \ \ \ \ &\norm{x_k}^2-\left(2-m\right)\gamma_{k}\big(f_{B_k}(x_k)-\underline{\mu_{B_k}}\big)+2\gamma_{k}\big(f_{B_k}(0)-\underline{\mu_{B_k}}\big).\notag\qedhere
	\end{align}
\end{proof}

\begin{theorem}\label{theorem3.8}
	Let $f_i:X\to\mathbb{R}$ be convex, $L_i-$smooth, and bounded below for all $i=~1,\dots,N$. Suppose that \ref{C2} is satisfied. Let $x_0\in X$. Let $\left(x_k\right)_{k\in\mathbb{N}}$ be a sequence generated by \cref{SGD} with stepsize $\left(\gamma_{k}\right)_{k\in\mathbb{N}}$ satisfying \cref{3.1} with $m<2$ and $0<\gamma_k\leq\gamma$. Then $\left(x_k\right)_{k\in\mathbb{N}}$ is bounded.
\end{theorem}

\begin{proof}
	Let 
	\begin{equation*}
		D:=\max_{\substack {B\subseteq\left\{1,\dots,N\right\}\\ |B|=b}}\big(f_{B}(0)-\underline{\mu_{B}}\big).
	\end{equation*}
	Let $I$ be the list of batches $B$ with $|B|=b$ such that $\lev{\underline{\mu_B}+\frac{2D}{2-m}}f_B\setminus\argmin f_B$ is nonempty. Assumption \ref{C2} implies
	\begin{equation}
		\left(\forall B\in I\right)\left(\exists M_B>0\right)\left(\forall x\in\lev{\underline{\mu_B}+\frac{2D}{2-m}}f_B\setminus\argmin f_B\right)\quad \norm{x}\leq M_B.\label{3.2}
	\end{equation}
	Let $M:=\max_{B\in {I}}M_B$, and let $x^*_B$ denote a minimizer for each batch function $f_B$. Define 
	\begin{equation*}
		c:=\max\left\{M,\frac{\gamma L_{\max}}{1+\gamma L_{\max}}\max_{\substack {B\subseteq\left\{1,\dots,N\right\}\\ |B|=b}}\norm{x_B^*}\right\}.
	\end{equation*}

	We will prove by induction that
	\begin{equation*}
		\left(\forall k\in\mathbb{N}\right)\quad\norm{x_k}^2\leq\max\left\{4c^2\left(1+\gamma L_{\max}\right)^2,\norm{x_0}^2\right\}.
	\end{equation*}
	The base case $k=0$ clearly holds. For the induction step, assume the induction hypothesis
	\begin{equation}
		\norm{x_k}^2\leq\max\left\{4c^2\left(1+\gamma L_{\max}\right)^2,\norm{x_0}^2\right\},\label{3.3}
	\end{equation}
	holds for some $k\in\mathbb{N}$. If
	\begin{equation*}
		\norm{x_{k+1}}\leq2c\left(1+\gamma L_{\max}\right),
	\end{equation*}
	then \cref{3.3} holds for $k+1$, and we are done. So, we consider the case where
	\begin{equation}
		\norm{x_{k+1}}>2c\left(1+\gamma L_{\max}\right).\label{3.4}
	\end{equation}

	By the triangle inequality, $L_{\max}$-smoothness of $f_{B_k}$ (see \cref{remark2.2}), and $\gamma_k\leq\gamma$, we get that
	\begin{align*}
		\norm{x_k}&\geq\norm{x_{k+1}}-\norm{x_{k+1}-x_k}\\
		&=\norm{x_{k+1}}-\gamma_k\norm{\nabla f_{B_k}(x_k)}\\
		&\geq\norm{x_{k+1}}-\gamma L_{\max}\norm{x_k-x^*_{B_k}}\\
		&\geq\norm{x_{k+1}}-\gamma L_{\max}\left(\norm{x_k}+\norm{x^*_{B_k}}\right).
	\end{align*}
	Rearrange and divide both sides by $1+\gamma L_{\max}$, we obtain
	\begin{equation}
		\norm{x_k}\geq\frac{\norm{x_{k+1}}}{1+\gamma L_{\max}}-\frac{\gamma L_{\max}}{1+\gamma L_{\max}}\norm{x_{B_k}^*}.\label{3.5}
	\end{equation}
	
	From the definition of $c$, we get that
	\begin{equation}
		\frac{\gamma L_{\max}}{1+\gamma L_{\max}}\norm{x_{B_k}^*}\leq c.\label{3.6}
	\end{equation}
	Combining \cref{3.5}, \cref{3.4}, \cref{3.6} and the definition of $c$, we obtain
	\begin{equation*}
		\norm{x_k}>2c-c=c\geq {M}.
	\end{equation*}
	Hence,
	\begin{equation*}
		\norm{x_k}>M.
	\end{equation*}
	This combined with \cref{3.2} and the definition of $M$ yields
	\begin{equation}
		\left(\forall B\in I\right)\quad x_k\notin \lev{\underline{\mu_B}+\frac{2D}{2-m}}f_B\setminus\argmin f_B.\label{3.7}
	\end{equation}
	
	From \cref{3.7} and the definition of $I$, we obtain
	\begin{equation*}
		\left(\forall B\subseteq\left\{1,\dots,N\right\},|B|=b\right)\quad x_k\notin \lev{\underline{\mu_B}+\frac{2D}{2-m}}f_B\setminus\argmin f_B,
	\end{equation*}
	which implies
	\begin{equation*}
		f_{B_k}(x_k)=\mu_{B_k}\text{ or }f_{B_k}(x_k)>\underline{\mu_{B_k}}+\frac{2D}{2-m}.
	\end{equation*}
	Since we assumed $x_k\notin\argmin f_{B_k}$ for all $k\in\mathbb{N}$, we deduce that
	\begin{equation}
		f_{B_k}(x_k)>\underline{\mu_{B_k}}+\frac{2D}{2-m}.\label{3.8}
	\end{equation}

	From Proposition \ref{proposition3.7}, \cref{3.3}, \cref{3.8} and the definition of $D$, we obtain
	\begin{align*}
		\norm{x^{k+1}}^2&\leq\norm{x^k}^2-(2-m)\gamma_{k}\big(f_{B_k}(x_k)-\underline{\mu_{B_k}}\big)+2\gamma_{k}\big(f_{B_k}(0)-\underline{\mu_{B_k}}\big)\\
		&< \max\left\{4c^2\left(1+\gamma L_{\max}\right)^2,\norm{x_0}^2\right\}-(2-m)\gamma_{k}\frac{2D}{2-m}+2\gamma_{k}D\\
		&=\max\left\{4c^2\left(1+\gamma L_{\max}\right)^2,\norm{x_0}^2\right\}.
	\end{align*}
	This conclude the proof.
\end{proof}

\begin{corollary}
	Let $f_i:X\to\mathbb{R}$ be convex, $L_i-$smooth, and bounded below for all $i=~1,\dots,N$. Suppose that \ref{C2} is satisfied. Let $\left(\gamma_k\right)_{k\in\mathbb{N}}$ be a sequence such that $0<\gamma_k\leq\gamma<\frac{1}{L_{\max}}$ for all $k\in\mathbb{N}$. Let $x_0\in X$ and let $\left(x_k\right)_{k\in\mathbb{N}}$ be a sequence generated by \cref{SGD}. Then $(x_k)_{k\in\mathbb{N}}$ is bounded.
\end{corollary}

\begin{proof}
	From \cref{remark3.5} and $0<\gamma_k\leq\gamma<\frac{1}{L_{\max}}$ for all $k\in\mathbb{N}$, it follows that the sequence $\left(\gamma_{k}\right)_{k\in\mathbb{N}}$ satisfies \cref{3.1} with $m=2L_{\max}\gamma<2$. Therefore, \cref{theorem3.8} applies and yields the conclusion.
\end{proof}

\begin{corollary}
	Let $f_i:X\to\mathbb{R}$ be convex, $L_i-$smooth, and bounded below for all $i=~1,\dots,N$. Suppose that \ref{C2} is satisfied. Let $\left(\lambda_k\right)_{k\in\mathbb{N}}$ be a decreasing sequence such that $0<\lambda_0<2$ and let $\lambda_{-1}:=\lambda_{0}$, $\gamma_{-1}>0$, $x_0\in X$. Let $\left(x_k\right)_{k\in\mathbb{N}}$ be a sequence generated by \cref{SGD} with \cref{DecSPS}. Then $(x_k)_{k\in\mathbb{N}}$ is bounded.
\end{corollary}

\begin{proof}
	From \cref{remark3.6} and $0<\lambda_0<2$, it follows that the sequence $\left(\gamma_{k}\right)_{k\in\mathbb{N}}$ satisfies \cref{3.1} with $m=\lambda_0<2$. Moreover, since $\left(\lambda_k\right)_{k\in\mathbb{N}}$ is a decreasing sequence, \cref{fact3.3} implies that $\gamma_k\leq\frac{\gamma_{-1}}{\lambda_0}\lambda_k\leq\gamma_{-1}$ for all $k\in\mathbb{N}$. Therefore, Theorem \ref{theorem3.8} applies and yields the conclusion.
\end{proof}

\subsection{Case \ref{C3}}
\label{s:discussion}
Let $f_i:=\frac{1}{2}d_{C_i}^2$ where $C_i$ is a nonempty closed convex subsets of $X$ and $d_{C}:=\inf_{c\in C}\norm{\cdot-c}$ for all $i=1,\dots,N$. Then each $f_i$ is convex, bounded below, and $1-$smooth with $\nabla f_i=\Id-P_{C_i}$ where $P_C$ is the orthogonal projector onto C and $\Id$ is the identity mapping on $X$.

In the case where $b=1$ and $\underline{\mu_i}=\mu_i=0$ for all $i=1,\dots,N$, \cref{DecSPS} implies
\begin{align*}
    \gamma_k&\overset{\text{def}}{=}\lambda_k\min\left\{\frac{f_{i_k}(x_k)-\mu_{i_k}}{\norm{\nabla f_{i_k}(x_k)}^2},\frac{\gamma_{k-1}}{\lambda_{k-1}}\right\}\\&=\lambda_k\min\left\{\frac{1}{2},\frac{\gamma_{k-1}}{\lambda_{k-1}}\right\}\\
    &\overset{\text{def}}{=}\lambda_k\min\left\{\frac{1}{2},\frac{\lambda_{k-1}\min\left\{\frac{1}{2},\frac{\gamma_{k-2}}{\lambda_{k-2}}\right\}}{\lambda_{k-1}}\right\}\\
    &=\lambda_k\min\left\{\frac{1}{2},\frac{\gamma_{k-2}}{\lambda_{k-2}}\right\}\\
    &\ \ \vdots\\
    &=\lambda_k\min\left\{\frac{1}{2},\frac{\gamma_{-1}}{\lambda_{0}}\right\}.
\end{align*}
Hence, \cref{SGD} with \cref{DecSPS} is equivalent to the relaxed random projection algorithm:
\begin{equation}
    x_{k+1}=x_k-\gamma_k\nabla f_{i_k}(x_k)=\left(1-\lambda_k\min\left\{\frac{1}{2},\frac{\gamma_{-1}}{\lambda_{0}}\right\}\right)x_k+\lambda_k\min\left\{\frac{1}{2},\frac{\gamma_{-1}}{\lambda_{0}}\right\}P_{C_{i_k}}x_k.\label{3.9}
\end{equation}

If each $C_i\subseteq X$ is an unbounded polyhedral set\footnote{Recall that a subset $C$ of $X$ is a polyhedral set if it is the intersection of finitely many closed halfspaces.}, then \cref{C3} is satisfied. By \cite[Theorem~2.5]{Bauschke2025}, if $\min\left\{\lambda_0,2\gamma_{-1}\right\}<4$, then the sequence $(x_k)_{k\in\mathbb{N}}$ generated by \cref{3.9} is bounded.

If we drop polyhedrality and let $\lambda_k=\gamma_{-1}=2$ for all $k\in\mathbb{N}$, then one can construct two sets $C_1,C_2$ such that \cref{C3} is satisfied and there exists a sampling of sets such that the sequence $(x_k)_{k\in\mathbb{N}}$ generated by \cref{3.9} blow up to infinity (see \cite[Example~4.1]{Bauschke2025}).

Even in this special case, we've seen that the sequence of iterates generated by \cref{SGD} with \cref{DecSPS} can either be bounded or unbounded. Thus, we leave the analysis of \cref{C3} in more settings as future work.

\section*{Data Availability}
No datasets were generated or analysed for the research described in this article.

\section*{Acknowledgments}
The research of HHB was supported by a Discovery Grant from the Natural Sciences and Engineering Research Council of Canada.

\end{document}